\documentclass[reqno,twoside,11pt]{amsart}

\usepackage{amsmath,amsfonts,calrsfs,fullpage,amssymb,color,verbatim,eucal,yfonts,mathrsfs}

\newtheorem{Theorem}{Theorem}[section]
\newtheorem{Definition}[Theorem]{Definition}
\newtheorem{Proposition}[Theorem]{Proposition}
\newtheorem{Lemma}[Theorem]{Lemma}

\makeatletter
\@addtoreset{equation}{section}

\makeatother

\def\R{\mathbb R}
\def\N{\mathbb N}

\def\ds{\displaystyle}

\title{\bf Schauder theorems for Ornstein-Uhlenbeck equations in infinite dimension}\date{}

\author[S. Cerrai]{Sandra Cerrai}
\address{Department of Mathematics\\
University of Maryland\\ 
College Park, MD 20742, USA}
\email{cerrai@math.umd.edu}

\author[A. Lunardi]{Alessandra Lunardi}
\address{
Dipartimento di Scienze Matematiche, Fisiche e Informatiche\\
Universit\`a di Parma\\
Parco Area delle Scienze, 53/A\\
43124 Parma, Italy}
\email{alessandra.lunardi@unipr.it}



\begin{document}

 \begin{abstract}  
We prove Schauder type estimates for stationary and evolution equations driven by the classical Ornstein-Uhlenbeck operator in a separable Banach space, endowed with a centered Gaussian measure. 
 \end{abstract}

 \maketitle
 
\section{Introduction}
 
Let $X$ be a separable Banach space, endowed with a centered Gaussian   measure $\gamma$, and let $H\subset X$ be the corresponding Cameron-Martin space.  In this context, an important differential operator that plays a central role in the Malliavin Calculus  is the classical Ornstein-Uhlenbeck operator, 
$${\mathcal L}u = \text{div}_{\gamma}\nabla_Hu, $$
where div$_{\gamma}$ is the Gaussian divergence and $\nabla_H$ is the gradient along $H$. It plays the role played by the Laplacian with respect to the Lebesgue measure in $\R^d$, being the operator associated to the quadratic Dirichlet form
$$(u,v)\mapsto \int_X \langle \nabla_Hu, \nabla_Hv\rangle_H\, d\gamma, \quad u, \;v\in W^{1,2}(X, \gamma). $$
The corresponding Markov semigroup is explicitly represented by
$$T(t)f(x) = \int_X f(e^{-t}x + \sqrt{1-e^{-2t}}y)\gamma (dy), \quad t>0, \; f\in C_b(X). $$
Its realization $T_p$ in $L^p(X, \gamma)$ is a contraction, strongly continuous semigroup  for every $p\in [1, +\infty)$, and it is 
analytic if $p>1$. In the latter case, the well known Meyer estimates imply that the domain of its infinitesimal generator $L_p$ coincides with the Sobolev space $W^{2,p}(X, \gamma)$. In particular, for every $\lambda >0$ and $f\in L^p(X, \gamma)$, the equation
\begin{equation}
\label{eq_risolvente}
\lambda u - {\mathcal L}u = f
\end{equation}
has a unique solution $u\in W^{2,p}(X, \gamma)$, and $\|u\|_{W^{2,p}(X, \gamma)} \leq C \|f\|_{L^p(X, \gamma)}$, with $C$ independent of $f$. See e.g. \cite[Ch. 5]{Boga} for a survey on Sobolev spaces with respect to Gaussian measures, and on the operators $L_p$. 

Here  we consider  a realization $L$ of ${\mathcal L}$  in the space  $C_b(X)$ of the continuous and bounded functions from $X$ to $\R$, whose resolvent $R(\lambda, L)$ is given, for $\lambda >0$, by 
$$R(\lambda, L)f (x) = \int_0^{\infty} e^{-\lambda t}T(t)f(x)\,dt, \quad f\in C_b(X). $$
The realizations of elliptic differential operators in spaces of continuous functions exhibit  typical difficulties. Even in finite dimension, the solution of \eqref{eq_risolvente} does not belong to  $C^2(\R^d)$ for general $f\in C_b(\R^d)$, while Schauder   theorems are available both for non-degenerate (\cite{DPL}) and for degenerate hypoelliptic (\cite{L}) Ornstein-Uhlenbeck operators. 

In our general setting we prove Schauder type regularity results for the  solution to  \eqref{eq_risolvente}, that are the H\"older counterpart of  the above mentioned maximal $L^p$ regularity results. The appropriate H\"older spaces (as well as the Sobolev spaces $W^{k,p}(X, \gamma)$) have to be chosen according to the structure of $\mathcal L$: indeed, it is well known that $T(t)$ and $R(\lambda, L)$ are smoothing operators   along the directions of the Cameron-Martin space $H$ only.  So, we use H\"older  spaces along $H$, defined as 
$$C^{\alpha}_H(X, Y):= \bigg\{f\in C_b(X, Y):\; [f]_{C^{\alpha}_H(X, Y)}:=\sup_{x\in X, \,h\in H\setminus\{0\}} \frac{\|f(x+h) - f(x)\|_Y}{\|h\|_H^{\alpha}}<+\infty \bigg\}, $$
$$\|f\|_{C^{\alpha}_H(X, Y)}:= \sup_{x\in X} \|f(x)\|_Y + [f]_{C^{\alpha}_H(X, Y)}, $$
for $\alpha\in (0,1)$ and for any Banach space $Y$. We prove that for every $f\in C^{\alpha}_H(X, \R)$ and for every $\lambda >0$, the unique solution $u\in D(L)$ of  \eqref{eq_risolvente} belongs to $C^{2+\alpha}_H(X, \R)$. This means that $u$ is twice continuously differentiable along $H$, it has bounded and continuous $H$-gradient $\nabla_Hu$ and $H$-Hessian operator $D^2_Hu$, with values respectively in $H$ and in the space of the bilinear quadratic forms  ${\mathcal L}^{(2)}(H)$, and $x\mapsto D^2_Hu(x)$ belongs to $C^{\alpha}_H(X, {\mathcal L}^{(2)}(H))$. Consequently, all the second order directional derivatives $\partial^2u/\partial h\partial k$ with $h$, $k\in H$ exist and belong to $C^{\alpha}_H(X, \R)$.

In the case that $f\in C_b(X)$ only, we prove that $u$  has bounded and continuous $H$-gradient $\nabla_Hu$, such that  
$$\sup_{x\in X, \,h\in H\setminus \{0\}} \frac{\|\nabla_Hu(x+2h)-2\nabla_Hu(x+h) + \nabla_Hu(x)\|_H}{\|h\|_H} <+\infty , $$
namely $\nabla_Hu$ satisfies a Zygmund condition along $H$. This is an infinite dimensional counterpart of the Zygmund regularity  of the gradients of solutions to elliptic differential equations in finite dimension. 
 
Schauder type regularity results are proved also  for the mild solution to  the Cauchy problem 
\begin{equation}
\label{Cauchy}\left\{ \begin{array}{l}
v_t(t,x)  = Lv(t,x)  + g(t,x), \quad t\in [0,T], \; x\in X, 
\\
\\
v(0, \cdot) = f, \end{array}\right. 
\end{equation}
namely for  the function 
\begin{equation}
\label{v}
v(t,x) = T(t)f(x) + \int_0^t T(t-s)g(s, \cdot)(x)ds, \quad  t\in [0,T], \; x\in X, 
\end{equation}
when $f\in C^{2+\alpha}_H(X, \R)$ and $g\in C_b([0,T]\times X; \R)$  such that $\sup_{t\in [0,T]} [g(t, \cdot)]_{C^{\alpha}_H(X; \R)} <+\infty$. 
However, while in finite dimension with non-degenerate $\gamma$ the function $v$ defined by \eqref{v} is a classical solution to \eqref{Cauchy}, in infinite dimension it is not differentiable with respect to $t$ in general, even if $g\equiv 0$. 

Our main interest  is  in the infinite dimensional case. However, if $X=\R^n$ the operator ${\mathcal L}$ reads as 
$${\mathcal L}u(x) = \text {Trace}\,[QD^2u(x)] - \langle x, \nabla u(x)\rangle$$
where $Q\geq 0$ is the covariance matrix of $\gamma$. If $Q>0$, namely if $\gamma$ is non-degenerate, our results are contained in \cite{DPL,L}. If $Q$ is not invertible the operator ${\mathcal L}$ is not hypoelliptic, and this paper provides new H\"older and Zygmund regularity results along the directions of the range of $Q$.

In infinite dimension, Schauder regularity results for elliptic equations driven by  Ornstein-Uhlenbeck operators 
 are already available in the case that  $X$ is a Hilbert space, $\gamma$ is non-degenerate, and the corresponding Ornstein-Uhlenbeck semigroup is smoothing in all directions (\cite{CDP,ABP,C}). Still in the case that $X$ is a Hilbert space and $\gamma$ is non-degenerate, Schauder regularity results for elliptic equations driven by the Gross Laplacian and some of its perturbations are also available (\cite{CDP1,ABP}).  See section \ref{section_biblio} for details.

\section{Notation and preliminaries}
 
Throughout the paper we use  notations, definitions and results of \cite{Boga} concerning Gaussian measures in Banach spaces. 

We consider a 
 separable Banach space $X$ endowed with a centered  Gaussian measure $\gamma$, and we denote by  $H$  the corresponding 
 Cameron-Martin space. It is isometric to the closure of $X^*$ in $L^2(X, \gamma)$, denoted by $X^*_{\gamma}$. The isometry $R_{\gamma}: X^*_{\gamma}\mapsto H $ is defined as follows: $R_{\gamma}f$ is the unique $y\in X$ such that $\int_X f(x) g(x) \gamma (dx) = g(y)$, for every $g\in X^*$. For every $h\in H$, $R_{\gamma}^{-1}h$ is usually denoted by $\hat{h}$. 
 
We recall the Cameron-Martin formula: for every $h\in H$, the translated measure $\gamma_h(B):=\gamma(B-h)$ is absolutely continuous with respect to $\gamma$, with density $\rho(x) = \exp{\hat{h}(x) -\|h\|^2_H/2}$. So, for every continuous and bounded $\varphi$ we have 
\begin{equation}
\label{CM}
\int_X \varphi(x+h)\gamma (dx) = \int_X \varphi (x)e^{\hat{h}(x) -\|h\|^2_H/2}\gamma(dx). 
\end{equation}
 We also  recall that for every $h\in H$, the function $\hat{h}$ is a Gaussian random variable with law ${\mathcal N}(0,\|h\|^2_H)$ in $\R$. Therefore, for every $p\in [1, +\infty)$ we have 
 \begin{equation}
 \label{legge}
 \| \hat{h} \|_{L^p(X, \gamma)} = \left( \frac{1}{\sqrt{2\pi} }\int_\R |\xi|^p \exp( -\xi^2/2)d\xi\right)^{1/p} \|h\|_H := k_p  \|h\|_H. 
 \end{equation}
 
A function $f:X\mapsto \R$ is called $H$-differentiable at $x\in X$ if there exists a (unique) linear bounded operator $\ell :H\mapsto \R$ such that 
$$\lim_{\|h\|_H \to 0} \frac{ f(x+h) - f(x) - \ell (h)}{\|h\|_H} =0. $$
We set $\ell := D_Hf(x)$. Since $H$ is a Hilbert space, there exists a unique $y\in H$ such that $D_Hf(x)( h) = \langle h, y\rangle _H$. Such $y$ is denoted by $\nabla_Hf(x)$. 

Since $H$ is continuously embedded in $X$, if $f$ is Frech\'et differentiable at $x$ it is also $H$-differentiable at $x$, and $f'(x)(h) = \langle \nabla_Hf(x), h\rangle_H$, for every $h\in H$. In particular, if $X$ is a Hilbert space,  $\gamma = {\mathcal N}(0,Q)$ is the centered Gaussian measure with covariance $Q$, and $\nabla f(x)$ is the gradient of $f$ at $x$, we have $\nabla_Hf(x) = Q\nabla f(x)$.

More generally, if $Y$ is a Banach space, a function $F:X\mapsto Y$ is called $H$-differentiable at $x\in X$ if there exists a linear bounded operator $L:H\mapsto Y$ such that 
$$Y-\lim_{\|h\|_H \to 0} \frac{ F(x+h) - F(x) - L(h)}{\|h\|_H} =0. $$
 $n$ times $H$-differentiable functions are defined by recurrence, in a canonical way.   Here we are interested in $n=2$, $3$. 
So, if $f$ is $H$-differentiable in $X$, we say that it is twice $H$-differentiable at $x$ if $D_H f :X\mapsto H'$ is differentiable at $x$, (equivalently, $\nabla_Hf :X\mapsto H$ is differentiable at $x$) and we 
define the Hessian operator $D^2_Hf(x)\in {\mathcal L}^{(2)}(H)$ (the space of the bounded bilinear forms from $H^2$ to $\R$), 
by $D^2_Hf(x)(k,h) := (Lh)(k)$, where $L$ is the operator in the definition, with $F(x)=D_Hf(x)$, $Y=H'$. Similarly, if $f$ is twice $H$-differentiable in $X$, we say that it is thrice $H$-differentiable at $x$ if $D^2_Hf : X\mapsto {\mathcal L}^{(2)}(H)$ is $H$-differentiable at $x$; in this case the third order derivative $D^3_Hf(x)\in {\mathcal L}^{(3)}(H)$ is defined as $D^3_Hf(x)(h,k,l) : = (Lh)(k,l)$, where $L$ is the operator in the definition, with now $F(x) = D^2_Hf(x)$, $Y=  {\mathcal L}^{(2)}(H)$.

\begin{Definition}
For $k\in \N$ we denote by $C^k_H(X)$  the subspace of $C_b(X)$ consisting of functions $k$ times $H$-differentiable at any point, with $D_H^j f$ continuous and bounded in 
${\mathcal L}^{(j)}(H)$ for $j\leq k$. $C^k_H(X)$ is endowed with the norm
$$\|f\|_{C^k_H(X)} := \sup_{x\in X}|f(x)| + \sum_{j=1}^k \sup_{x\in X} \| D^j_Hf(x)\|_{{\mathcal L}^{(j)}(H)}. $$
\end{Definition}

The Ornstein--Uhlenbeck semigroup is defined by 
\begin{equation}
\label{OU}
T(t) f(x) := \int_X f(e^{-t}x+ \sqrt{1-e^{-2t}}y) \gamma(dy), \quad t>0, \; f\in C_b(X).
\end{equation}
Then $T(t)$ maps $C_b(X)$ into itself for every $t>0$, and  
\begin{equation}
\label{stimasup}
\|T(t) f\|_{\infty} \leq \|f\|_{\infty}, \quad t>0, \; f\in C_b(X).
\end{equation}
Neverthless, $T(t)$ is
not  strongly continuous in $C_b(X)$, and not even in 
the subspace $BUC(X)$ of the bounded and uniformly continuous
functions.  Indeed, for  $f\in BUC(X)$ it is easy to see that  
\[
\lim_{t\to 0^+} \|T(t)f-f\|_{\infty} =0 \Longleftrightarrow \lim_{t\to 0^+} \|f(e^{-t}\cdot) - f\|_{\infty} =0.
\]
However, for every fixed $x\in X$ the function $t\mapsto T(t)f(x)$ is continuous in $[0, +\infty)$ by the Do\-minated Convergence Theorem. It follows that  for every $ \lambda >0$  the     linear operator $F(\lambda)$ defined by   
$$F(\lambda)f(x) := \int_0^{+\infty} e^{-\lambda t} T(t)f(x) \,dt, \quad \lambda >0, \; f\in C_b(X), \; x\in X, $$
belongs to $ {\mathcal L}(C_b(X))$ and it is  one to one. Moreover, since $T(t)$ is a semigroup, the family $\{ F(\lambda):\; \lambda >0\}$ satisfies
  the resolvent identity. Therefore there exists a linear  operator $L:D(L) \mapsto X$ such that $F(\lambda)= R(\lambda, L)$ for every $\lambda >0$. 

The operator $L$ is called {\em generator } of $T(t)$ in $C_b(X)$, although it is not an infinitesimal generator in the usual sense. 
So, as in the case of strongly continuous semigroups, we have
\begin{equation}
\label{risolvente}
(R(\lambda, L) f)(x) =  \int_0^{+\infty} e^{-\lambda t} T(t)f(x) \,dt, \quad \lambda >0, \; f\in C_b(X), \; x\in X, 
\end{equation}
and by \eqref{stimasup} we obtain
\begin{equation}
\label{stimasup_res}
\|R(\lambda, L) f\|_{\infty} \leq \|f\|_{\infty}, \quad \lambda>0, \; f\in C_b(X). 
\end{equation}

Let us recall that the realization $T_p(t)$ of $T(t)$ in $L^p(X, \gamma)$ is a strongly continuous, contraction,  analytic semigroup, for every $p\in (1, +\infty)$. The domain of its infinitesimal generator $L_p$ is equal to the Sobolev space $W^{2,p}(X, \gamma)$, and  the graph norm of $L_p$ is equivalent to the Sobolev norm. Moreover, 
$$L_p u = \text{div}_\gamma \nabla_Hu = \sum_{j=1}^{\infty} \left( \frac{\partial}{\partial h_j} - \hat{h}_j\right) \frac{\partial u}{\partial h_j}, $$
where div$_{\gamma}$ is the Gaussian divergence, $\{h_j:\; j\in \N\}$ is any orthonormal basis of $H$, and the series converges in $L^p(X, \gamma)$. See e.g. \cite[Ch. 5]{Boga}. 
If $X$ is a Hilbert space,  $\gamma$ is a non-degenerate centered Gaussian measure with covariance $Q$,  and $\{e_j:\; j\in \N\}$ is any orthonormal basis of $X$ consisting of eigenvectors of $Q$, $Qe_j = \lambda_j e_j$, then $\{\sqrt{\lambda_j} e_j:\; j\in \N\}$ is an orthonormal basis of $H$ and the 
 above series reads as 
$$L_p u(x)  =   \sum_{j=1}^{\infty} \bigg(\lambda_j \frac{\partial^2 u}{\partial e_j^2}(x) -  x_j \frac{\partial u}{\partial e_j}(x)\bigg), $$
where $x_j:= \langle x, e_j\rangle$. 

Using the characterizations $D(L_p) =W^{2,p}(X, \gamma)$ for $p>1$, we obtain a characterization of $D(L)$, as follows. 

\begin{Lemma}
$$D(L)= \{u\in \bigcap_{p>1}W^{2,p}(X, \gamma):\; u, \; {\mathcal L}u \in C_b(X)\} = \{u\in \bigcup_{p>1}W^{2,p}(X, \gamma):\; u, \; {\mathcal L}u \in C_b(X)\} . $$
Moreover, for every $u\in D(L)$, $Lu$ is a continuous and bounded version of $ {\rm div}_{\gamma} \nabla_Hu$. 
\end{Lemma}
\begin{proof}
For $u\in D(L)$ and $\lambda >0$ set $f:= \lambda u - Lu$, so that $u$ is given by \eqref{risolvente}. Since  $T_p(t)$ agrees with $T(t)$ on $C_b(X)$ for every $p>1$, we have $u= \int_0^{+\infty} e^{-\lambda t} T_p(t)f \,dt = R(\lambda, L_p)f$. Therefore, $u\in W^{2,p}(X, \gamma)$ for every $p>1$ and   $Lu = L_pu$, $\gamma$-a.e. So, $Lu$ is a continuous and bounded version of $L_pu =$ div$_{\gamma} \nabla_Hu$. 

Conversely, if $u\in W^{2,p}(X, \gamma)$ for some $p>1$ we have $u= \int_0^{+\infty} e^{-\lambda t} T_p(t)(\lambda u -L_pu ) \,dt$ for every $\lambda >0$. If  $u$, ${\mathcal L}u =L_pu \in C_b(X)$ we obtain $u = R(\lambda , L)f$, with $f=\lambda u - {\mathcal L}u $, and therefore $u\in D(L)$. 
\end{proof}


The following smoothing properties are easily shown. 

\begin{Proposition}
\label{Pr:C-C1} For every $f\in C_b(X)$ and $t>0$, $T(t)f$ is infinitely times $H$-differentiable at every $x\in X$. Setting
$$c(t):=   \frac{e^{-t}}{\sqrt{1-e^{-2t}}} , \quad t>0, $$
 we have
\begin{equation}
\label{funzionegradiente}
D_HT(t)f(x)(h) = \langle \nabla_H T(t)f(x), h\rangle _H = c(t) \int_X f(e^{-t}x + \sqrt{1-e^{-2t}}y) \hat{h}(y) \gamma(dy), 
\end{equation}
\begin{equation}
\label{funzionederseconde}
D^2_H T(t)f(x)(h,k)  =  c(t)^2  \int_X f(e^{-t}x + \sqrt{1-e^{-2t}}y)(  \hat{h}(y)\hat{k}(y) - \langle h, k\rangle_H) \gamma(dy), 
\end{equation}
\begin{equation}
\label{funzionederterze}
\begin{array}{lll}
D^3_H T(t)f(x)(h,k,l) & = & \ds   - c(t)^3  \int_X f(e^{-t}x + \sqrt{1-e^{-2t}}y)(  \hat{l}(y)  \langle h, k\rangle_H +  \hat{h}(y)  \langle k, l\rangle_H
+  \hat{k}(y)  \langle h, l\rangle_H) \gamma(dy)
\\
\\
& &  + c(t)^3  \int_X f(e^{-t}x + \sqrt{1-e^{-2t}}y) \hat{h}(y)\hat{k}(y) \hat{l}(y) \, \gamma(dy),  
\end{array}
\end{equation}
for every $h$, $k$, $l\in H$. The function $(t,x)\mapsto T(t)f(x)$ is continuous in $[0, +\infty)\times X$, and the functions $(t,x)\mapsto D_H^jT(t)f(x)$ ($j=1, 2, 3$) are continuos in $(0, +\infty)\times X$, with values in  
${\mathcal L}^{(j)}(X)$, respectively. Moreover, 
 for every $x\in X$ and $t>0$ we have
\begin{equation}
\label{stimagradienteH}
\begin{array}{ll}
(i) & |\nabla_H T(t)f(x)|_H \leq  c(t)   \|f\|_{\infty}, 
\\
\\
(ii) & \| D^2_H T(t)f(x)\|_{{\mathcal L}^{(2)}(H)} \leq  2c(t) ^2   \|f\|_{\infty}, 
\\
\\
(iii) & \| D^3_H T(t)f(x)\|_{{\mathcal L}^{(3)}(H)} \leq  (3 +k_3^3) c(t) ^3   \|f\|_{\infty}, 
\end{array}
\end{equation}
\end{Proposition}
\begin{proof} 
Formulae \eqref{funzionegradiente}, \eqref{funzionederseconde}, \eqref{funzionederterze} are easily proved using the Cameron-Martin formula. For instance concerning \eqref{funzionegradiente}, using \eqref{CM} we get  
$$T(t)f(x+h) - T(t)f(x) = \int_X f(e^{-t}x + \sqrt{1-e^{-2t}}y) [ \exp ( c(t)\hat{h}(y) - c(t)^2 \|h\|^2_H/2 ) -1]\,\gamma(dy)$$
which yields  \eqref{funzionegradiente}. \eqref{funzionederseconde}, \eqref{funzionederterze} are proved in the same way. Estimates \eqref{stimagradienteH} are consequence of \eqref{funzionegradiente}, \eqref{funzionederseconde}, \eqref{funzionederterze} through the H\"older inequality and \eqref{legge} (in particular, the constant $k_3^3$ in the right hand side of \eqref{stimagradienteH} comes from estimating $\|\hat{h} \hat{k}  \hat{l}\|_{L^1(X, \gamma)} \leq 
\|\hat{h}\|_{L^3(X, \gamma)} \|\hat{k} \|_{L^3(X, \gamma)}\| \hat{l}\|_{L^3(X, \gamma)}$). 
Also the continuity of $(t,x)\mapsto T(t)f(x)$ and of $D_H^jT(t)f(x)$ for $j=1, 2, 3$ is a consequence of the respective representation formulae, through the Dominated Convergence Theorem. 
\end{proof}

 For functions in $C^1_H(X)$ the estimates in \eqref{stimagradienteH} may be improved. The proof is similar, and it is omitted.

\begin{Proposition}
\label{Pr:C1b}
For every $f\in C^1_H(X)$,    for any $t\geq 0$, and for every $x\in X$ we have
\begin{equation}
\label{derivataT(t)f}
\langle \nabla_H T(t)f(x), h \rangle_H =  e^{-t}\int_X\langle \nabla_H f(e^{-t}x+ \sqrt{1-e^{-2t}y}), h\rangle_H\, \gamma (dy), 
\end{equation}
\begin{equation}
\label{derivatasecondaT(t)f}
\langle D^2_H T(t)f(x)(h,k)   =  e^{-t}\int_X\langle \nabla_H f(e^{-t}x+ \sqrt{1-e^{-2t}y}), h\rangle_H\hat{k}(y) \, \gamma (dy), 
\end{equation}
\begin{equation}
\label{derivataterzaT(t)f}
\langle D^3_H T(t)f(x)(h,k,l)   =  e^{-t}\int_X\langle \nabla_H f(e^{-t}x+ \sqrt{1-e^{-2t}y}), h\rangle_H( \hat{k}(y)\hat{l}(y) - \langle k,l\rangle_H) \, \gamma (dy). 
\end{equation}
The function $(t,x)\mapsto \nabla_H T(t)f(x)$ is continuous in $[0, +\infty)\times X$ with values in $H$, and for every $x\in X$ and $t>0$ we have 
\begin{equation}
\label{stimederivate}
\begin{array}{ll}
(i) & \|\nabla_H T(t)f(x)\|_H \leq  \|\nabla_H f\|_{\infty}, 
\\
\\
(ii) & \| D^2_H T(t)f(x)\|_{{\mathcal L}^{(2)}(H)} \leq  c(t)   \|\nabla_Hf\|_{\infty}, 
\\
\\
(iii) & \| D^3_H T(t)f(x)\|_{{\mathcal L}^{(3)}(H)} \leq  2c(t) ^2   \|\nabla_Hf\|_{\infty}. 
\end{array}
\end{equation}
\end{Proposition}

\section{H\"older spaces and Schauder type theorems}

We introduce a class of H\"older spaces that arise ``naturally" in  this setting. 

\begin{Definition}
If $Y$ is any Banach space and $\alpha\in (0,1)$, the space $C^{\alpha}_H(X, Y)$ is the subspace of $C_b(X, Y)$ consisting of the functions $F$ such that 
$$[F]_{\alpha}:= \sup_{h\in H\setminus\{0\}, \,x\in X} \frac{\| f(x+h) - f(x) \|_Y}{\|h\|_H^{\alpha}} <+\infty . $$
$C^{\alpha}_H(X, Y)$ is normed by 
$$\|F\|_{C^{\alpha}_H(X, Y)} := \|F\|_{\infty} + [F]_{\alpha}. $$
If $Y=\R$ the space $C^{\alpha}_H(X, \R)$ is denoted by $C^{\alpha}_H(X)$. Moreover, we denote by $C^{1+\alpha}_H(X)$, $C^{2+\alpha}_H(X)$ the subspaces  of $C^{1}_H(X)$, $C^{2}_H(X)$, consisting of functions $f$ such that $D_Hf \in C^{\alpha}_H(X, {\mathcal L} (H))$, $D^2_Hf \in C^{\alpha}_H(X, {\mathcal L}^{(2)}(H))$, respectively. They are endowed with the norms
$$\|f\|_{C^{1+\alpha}_H(X )} := \|f\|_{C^{1}_H(X)} + [D_Hf]_{\alpha} = \|f\|_{C^{1}_H(X)} + \sup_{h\in H\setminus\{0\}, \,x\in X} \frac{\| D_Hf(x+h) - D_Hf(x) \|_{{\mathcal L} (H)}} {\|h\|_H^{\alpha}} 
$$
$$\|f\|_{C^{2+\alpha}_H(X )} := \|f\|_{C^{2}_H(X)} + [D^2_Hf]_{\alpha} = \|f\|_{C^{2}_H(X)} + \sup_{h\in H\setminus\{0\}, \,x\in X} \frac{\| D^2_Hf(x+h) - D^2_Hf(x) \|_{{\mathcal L}^{(2)}(H)}} {\|h\|_H^{\alpha}} 
$$
\end{Definition}

 The behavior of the semigroup $T(t)$ in the space $C^{k+\alpha}_H(X)$, $k=0, 1, 2$,  is similar to the one in $C_b(X)$. Below, we just state   the properties that will be used in the sequel. 
 
 \begin{Lemma}
 \label{Le:T(t)holder}
$T(t)\in {\mathcal L}(C^{k+\alpha}_H(X))$ for  every $t>0$,  $k=0, 1, 2$,  $\alpha\in (0, 1)$, and there are $C_k>0$ such that 
\begin{equation}
\label{sgrkalpha}
\|T(t)\|_{ {\mathcal L}(C^{k+\alpha}_H(X ))} \leq  1, \quad t>0. 
\end{equation}
Moreover, we have 
\begin{equation}
\label{sgralpha}
[T(t)f]_{\alpha} \leq e^{-\alpha t}[f]_{\alpha}, \quad  t>0, \; f\in C^{\alpha}_H(X), 
  \end{equation}
\begin{equation}
\label{gradsgralpha}
  [D_HT(t)f]_{C^{\alpha}_H(X, H')} \leq e^{-\alpha t} c(t)  [f]_{\alpha}, \quad  t>0,  \; f\in C^{\alpha}_H(X), 
  \end{equation}
whereas 
\begin{equation}
\label{gradsgrzeroalpha}
  [D_HT(t)f]_{C^{\alpha}_H(X, H')} \leq 2e^{-\alpha t} c(t)^{1+\alpha}  \|f\|_{\infty}, \quad  t>0,  \; f\in C_b(X). 
  \end{equation}
 \end{Lemma}
 \begin{proof}
 Let $t>0$ and $f\in C^{\alpha}_H(X)$.  For every $h\in H$ we have
 $$\begin{array}{lll}
|T(t)f(x+h) - T(t)f(x)| & = & \ds \bigg| \int_X[ f(e^{-t}(x+h)+ \sqrt{1-e^{-2t}}y) -  f(e^{-t}x+ \sqrt{1-e^{-2t}}y)] \gamma(dy) \bigg|
\\
\\
&  \leq & (e^{-t}\|h\|_H)^{\alpha}[f]_{\alpha}, 
\end{array}$$
which yields \eqref{sgralpha}.    \eqref{sgrkalpha} follows, for $k=0$. 

If $f\in C^{1+\alpha}_H(X)$,   $T(t)f \in C^1_H(X)$ by Proposition \ref{Pr:C1b}, and estimates \eqref{stimasup} and 
\eqref{stimederivate}(i) yield 
$$\|T(t)f\|_{C^1_H(X)} \leq \|f\|_{C^1_H(X)}. $$
By  \eqref{derivataT(t)f}, for every $t>0$, $x\in X$  we have
 $$D_HT(t)f(x) = e^{-t} \int_X D_H  f(e^{-t}x+ \sqrt{1-e^{-2t}}y) \gamma(dy), $$
 so that   for each $h$, $k\in H$ we have 
 $$\begin{array}{l}
 |(D_HT(t)f(x+h) - D_HT(t)f(x))(k)| = 
 \\
 \\
 = \ds e^{-t} \bigg| \int_X[ D_Hf(e^{-t}(x+h)+ \sqrt{1-e^{-2t}}y) -  D_Hf(e^{-t}x+ \sqrt{1-e^{-2t}}y)](k) \gamma(dy) \bigg| 
 \\
 \\
 \leq e^{-t}(e^{-t}\|h\|_H)^{\alpha}[D_Hf]_{C^{\alpha}_H(X, H')}\|k\|_H, 
 \end{array}$$
and  \eqref{sgrkalpha} follows,  for $k=1$. The statement for $k=2$ is proved  in the same way.  

Let us prove \eqref{gradsgralpha}. Let  $f\in C^{\alpha}_H(X)$. 
By \eqref{funzionegradiente}, for every $h$, $k\in H$ we have
$$\begin{array}{l}
 |(D_H T(t)f(x+h)- D_H T(t)f(x))(k) |
  \\
  \\
 =  \ds c(t) \bigg| \int_X [f(e^{-t}(x+h) + \sqrt{1-e^{-2t}}y)-f(e^{-t}x+ \sqrt{1-e^{-2t}}y)] \hat{k}(y) \gamma(dy)\bigg|
  \\
  \\
 \leq  \ds c(t)(e^{-t}\|h\|_H)^{\alpha}\|\hat{k}\|_{L^1(X, \gamma)} [f]_{\alpha}
 \end{array} $$
 and \eqref{gradsgralpha} follows, recalling that $\|\hat{k}\|_{L^1(X, \gamma)}\leq \|\hat{k}\|_{L^2(X, \gamma)} = \| k\|_H$. 
 
Estimate \eqref{gradsgrzeroalpha} 
follows combining \eqref{stimagradienteH}(i)-(ii): indeed, for every $t>0$, $x\in X$, $h\in H$ we have 
$$\|D_HT(t)f(x+h) - D_HT(t)f(x) \|_{H'} \leq 2 c(t) \|f\|_{\infty}$$
by \eqref{stimagradienteH}(i), and 
$$\|D_HT(t)f(x+h) - D_HT(t)f(x) \|_{H'} \leq 2 c(t)^2\|h\|_H  \|f\|_{\infty}$$
by \eqref{stimagradienteH}(ii). Therefore, 
$$\|D_HT(t)f(x+h) - D_HT(t)f(x) \|_{H'} \leq (2 c(t))^{1-\alpha} ( 2 c(t)^2\|h\|_H  )^{\alpha} \|f\|_{\infty}$$
and \eqref{gradsgrzeroalpha}  is proved. 
\end{proof}

 The key estimates in what follows are in the next lemma. 
 
 \begin{Lemma}
 \label{Le:fund}
 For every $\alpha\in (0,1)$ there is $C_{1,\alpha}>0$ such that 
 \begin{equation}
 \label{stimagradientealpha}
 \|\nabla_H T(t)f(x)\|_H \leq \frac{C_{1,\alpha}}{t^{(1-\alpha)/2}}\| f \|_{C^{\alpha}_H(X)}, \quad t>0, \; f\in C^{\alpha}_H(X), \;x\in X. 
 \end{equation}
 Consequently, there are $C_{2,\alpha}$, $C_{3,\alpha}>0$ such that 
 \begin{equation}
 \label{stimaderivatealpha}
\begin{array}{lll}
(i) &  \|D^2_H T(t)f(x)\|_{{\mathcal L}^{(2)}(H)}  \leq \ds \frac{C_{2,\alpha}}{t^{1-\alpha/2}}\| f \|_{C^{\alpha}_H(X)}, &  t>0, \; f\in C^{\alpha}_H(X), \;x\in X,
\\
\\
(ii)&  \|D^3_H T(t)f(x)\|_{{\mathcal L}^{(3)}(H)}  \leq \ds \frac{C_{3,\alpha}}{t^{3/2-\alpha/2}}\| f \|_{C^{\alpha}_H(X)}, & t>0, \; f\in C^{\alpha}_H(X), \;x\in X. 
\end{array}
 \end{equation}
 \end{Lemma}
 \begin{proof}
 Let $t>0$, $f\in C^{\alpha}_H(X)$, $h\in H\setminus \{0\}$. For every $s>0 $ we have 
 $$\begin{array}{lll}
 |\langle \nabla_H T(t)f(x), h\rangle _H| & \leq & \ds  \left| \langle \nabla_H T(t)f(x), h\rangle _H- \frac{T(t)f(x+sh) - T(t)f(x)}{s}\right| 
 \\
 \\
 & & \ds + \left|  \frac{T(t)f(x+sh) - T(t)f(x)}{s}\right| 
 \\
 \\
&  =: & I_1(s)+I_2(s). \end{array}$$
 Using \eqref{gradsgralpha} we get
 $$\begin{array}{lll}
|  I_1(s) | & = & \ds  \bigg| \frac{1}{s} \int_0^s \bigg( \langle \nabla_H T(t)f(x+ \sigma h), h\rangle _H - \langle \nabla_H T(t)f(x), h\rangle _H\bigg) d\sigma  \bigg| 
\\
\\
&\leq &  \ds \frac{1}{s} \int_0^s  c(t) \sigma^{\alpha} \|h\|_H^{\alpha +1}  [f]_{\alpha} d\sigma = \frac{ 1}{\alpha +1}c(t)  s^{\alpha} [f]_{\alpha}, 
\end{array}$$
while using \eqref{sgralpha} we get
$$|  I_2(s) | \leq s^{\alpha -1}  \|h\|_H^{\alpha }  [f]_{\alpha} . $$
Choosing now $s=  t^{1/2}/ \|h\|_H$ we obtain
$$|\langle \nabla_H T(t)f(x), h\rangle _H| \leq \bigg( \frac{ 1}{\alpha +1}c(t)  t^{\alpha/2}  +  t^{(\alpha-1)/2}\bigg)  \|h\|_H  [f]_{\alpha}, $$
and this yields \eqref{stimagradientealpha}. 

To prove \eqref{stimaderivatealpha}  it is sufficient to split $D^2_H T(t)f = D^2_H T(t/2) T(t/2)f$, $D^3_H T(t)f = D^3_H T(t/2) T(t/2)f$, and to use estimates  \eqref{stimagradientealpha} and \eqref{stimederivate}(ii) and (iii).
\end{proof}

\begin{Theorem}   
\label{Schauderell}       
Let $\lambda>0$, $f\in C^{\alpha}_H(X)$ with $0<\alpha <1$. Then the unique solution to
$$\lambda u - Lu =f$$
belongs to $C^{2+\alpha}_H(X)$, and    there is $C = C(\lambda, \alpha)>0$ such that 
\begin{equation}
\label{Schauder}
\|u\|_{C^{2+\alpha}_H(X)} \leq C \|f\|_{   C^{\alpha}_H(X)}. 
\end{equation}
\end{Theorem}
\begin{proof}       
Recalling  that $u$ is given by the representation formula \eqref{risolvente} it is not difficult to see that $u\in C^{2 }_H(X)$, and that 
\begin{equation}
\label{gradu}
\nabla_Hu(x) =  \int_0^{+\infty} e^{-\lambda t} \nabla_HT(t)f(x) \,dt, 
\end{equation}
\begin{equation}
\label{D^2u}
D^2_Hu(x) =  \int_0^{+\infty} e^{-\lambda t} D^2_HT(t)f(x) \,dt. 
\end{equation}
Notice that the right-hand sides of \eqref{gradu} and \eqref{D^2u} are meaningful, since $t\mapsto  \nabla_HT(t)f(x)$, $t\mapsto D^2_HT(t)f(x)$, 
are continuous for $t>0$ with values in $H$, ${\mathcal L}^{(2)}(H)$, respectively, 
by Proposition \ref{Pr:C-C1}, and their norms are bounded by $C_{1,\alpha}  t^{-1/2+\alpha/2}\| f \|_{C^{\alpha}_H(X)}$ , $C_{2,\alpha}  t^{-1+\alpha/2}\| f \|_{C^{\alpha}_H(X)}$, respectively, by Lemma \ref{Le:fund}.  
Then, \eqref{gradu} and \eqref{D^2u} follow in a standard way. They yield that $\nabla_Hu$, $D^2_Hu$ are continuous and bounded, with 
\begin{equation}
\label{Schauder0}  
\begin{array}{l}
\|  \nabla_Hu(x)\|_H \leq C_{1,\alpha}\lambda^{-1/2 -\alpha/2}\Gamma(1/2+\alpha/2)\|f\|_{C^{\alpha}_H(X)},
\\
\\
 \| D^2_Hu(x)\|_{{\mathcal L}^{(2)}(H)} \leq C_{2,\alpha}\lambda^{  -\alpha/2}\Gamma( \alpha/2)\|f\|_{C^{\alpha}_H(X)}, 
 \end{array}
\end{equation}
for every $x$, where $\Gamma (\theta) = \int_0^{\infty} e^{-t}t^{\theta -1}dt$ is the Euler function, and the constants $C_{1,\alpha}$, $C_{2,\alpha}$ are given by \eqref{stimaderivatealpha}.  

To prove that $D^2_Hu\in C^{\alpha}_H(X, {\mathcal L}^{(2)}(H))$ we use an interpolation argument. For every  $x\in X$ and $h\in H$   we split 
$D^2_Hu(x+h) - D^2_Hu(x)$ as $a(x+h) - a(x) + b(x+h) - b(x)$, where
\begin{equation}
\label{a,b}
a(y): =  \int_0^{\|h\|^2_H} e^{-\lambda t} D^2_HT(t)f(y) \,dt, \quad b(y): =  \int_{\|h\|^2_H}^{\infty} e^{-\lambda t} D^2_HT(t)f(y) \,dt. 
\end{equation}
Then, 
$$\|a(x+h) - a(x)\|_{{\mathcal L}^{(2)}(H)} \leq  \int_0^{\|h\|^2_H} e^{-\lambda t} \| D^2_HT(t)f(x+h) -  D^2_HT(t)f(x)\|_{{\mathcal L}^{(2)}(H)} dt, $$
where, for every $t>0$, 
$$ \| D^2_HT(t)f(x+h) -  D^2_HT(t)f(x)\|_{{\mathcal L}^{(2)}(H)} \leq 2 \sup_{y\in X}\|D^2_HT(t)f(y)\|_{{\mathcal L}^{(2)}(H)}\leq 2C_{2, \alpha}t^{-1+\alpha/2}\|f\|_{C^{\alpha}_H(X)}, $$
by \eqref{stimaderivatealpha}(i). Therefore, 
$$\|a(x+h) - a(x)\|_{{\mathcal L}^{(2)}(H)} \leq \frac{4C_{2,\alpha}}{\alpha} \|f\|_{C^{\alpha}_H(X)}\|h\|_H^{\alpha}. $$
Moreover, 
$$\|b(x+h) - b(x)\|_{{\mathcal L}^{(2)}(H)} \leq  \int_{\|h\|^2_H}^{\infty} e^{-\lambda t} \| D^2_HT(t)f(x+h) -  D^2_HT(t)f(x)\|_{{\mathcal L}^{(2)}(H)} dt, $$
where, for every $t>0$,  
$$\begin{array}{lll}
 \| D^2_HT(t)f(x+h) -  D^2_HT(t)f(x)\|_{{\mathcal L}^{(2)}(H)} & = &  \ds  \| \int_0^1 D^3_HT(t)f(x+\sigma h)(h, \cdot, \cdot) d\sigma\|_{{\mathcal L}^{(2)}(H)}
\\
\\
& \leq  &  C_{3,\alpha}t^{-3/2 +\alpha/2}\|f\|_{C^{\alpha}_H(X)}\|h\|_H, 
\end{array}$$
by \eqref{stimaderivatealpha}(ii). Therefore, 
$$\|b(x+h) - b(x)\|_{{\mathcal L}^{(2)}(H)} \leq \frac{2C_{3,\alpha}}{1-\alpha} \|f\|_{C^{\alpha}_H(X)}\|h\|_H^{\alpha}. $$
Summing up we obtain that $D^2_Hu$ is $H$-H\"older continuous and 
$$[D^2_Hu]_{C^{ \alpha}_H(X, {\mathcal L}^{(2)}(H))} \leq  \left( \frac{4C_{2,\alpha}}{\alpha} +  \frac{2C_{3,\alpha}}{1-\alpha} \right)  \|f\|_{C^{\alpha}_H(X)}. $$
Such estimate and \eqref{stimasup_res}, \eqref{Schauder0} yield  \eqref{Schauder}. \end{proof}

The procedure of Theorem \ref{Schauderell}   fails for $\alpha =0$ from the very beginning, since the (optimal) estimate   $\|D^2T(t)f\|_{{\mathcal L}^{(2)}(H)}$ $\leq ct^{-1}\|f\|_{\infty}$ is not enough to guarantee that  the right hand side of \eqref{D^2u} is  meaningful for general $f\in C_b(X)$. This is not due to our technique, but to the general lack of maximal regularity results for elliptic differential equations in spaces of continuous functions:
even in finite dimension it is   known that the domain of Ornstein-Uhlenbeck operators (as well as the domains of the Laplacian and of other second order elliptic differential operators)  in $C_b(\R^d)$ is not contained in $C^2(\R^d)$, for $d\geq 2$.

Of course, estimate \eqref{gradsgrzeroalpha} 
and the procedure  of Theorem \ref{Schauderell} give, for $u = R(\lambda, L) f$, 
$$\nabla_H u \in C^{\theta}_H(X, H), \quad \|\nabla_H u\|_{C^{\theta}_H(X, H)}\leq K \|f\|_{\infty}, $$
for every $\theta\in (0, 1)$, with $K= K(\lambda, \theta)$ independent of $f$. However, $K(\lambda, \theta)$ blows up as $\theta $ goes to $1$. 

Still, a modification of Theorem \ref{Schauderell}  gives an embedding of the domain of $L$ that is similar to known embeddings in the finite dimensional case. 
To this aim we have to introduce Zygmund spaces along $H$, as follows.

\begin{Definition}
If $Y$ is any Banach space, we denote by $Z_H(X, Y)$ the set of continuous and bounded functions $F:X\mapsto Y$ such that 
\begin{equation}
\label{condZygmund}
[F]_{Z_H(X, Y)} := \sup_{x\in X, \,h\in H\setminus \{0\}}\frac{\| F(x+2h)-2F(x+h) + F(x)\|_Y}{\|h\|_H} <+\infty . 
\end{equation}
$Z_H(X, Y)$ is normed by 
$$\|F\|_{Z_H(X, Y)} :=  \sup_{x\in X}\|F(x)\|_Y + [F]_{Z_H(X, Y)} . $$
\end{Definition}

It is easy to see that continuous and bounded $H$-Lipschitz functions from $X$ to $Y$ belong to $Z_H(X, Y)$. Even in the one dimensional case (with $X=Y=H=\R$), there are continuous and bounded functions satisfying condition \eqref{condZygmund} that are not locally Lipschitz continuous. 
          
\begin{Theorem}   
\label{Th:Zygmund}       
Let $\lambda>0$, $f\in C_{b}(X)$. Then the unique solution to
$$\lambda u - Lu =f$$
satisfies  $\nabla_Hu \in Z_H(X, H)$. 
Moreover  there is $C  >0$ such that 
\begin{equation}
\label{Zygmund}
\|\nabla_Hu\|_{Z_H(X, H)}  \leq C \|f\|_{ \infty}. 
\end{equation}
\end{Theorem}
\begin{proof}  We already know that  $u\in C^{1+\theta}_H(X)$ for every $\theta \in (0,1)$ by the above considerations; in particular $\nabla_Hu$ is continuous and bounded. 

To prove that $\nabla_Hu \in Z_H(X, H)$, for every $h\in H$ we consider the functions $a$ and $b$ defined in \eqref{a,b}. Using  \eqref{stimagradienteH}(i) we get, for every $x\in X$,  
$$\begin{array}{l}
\| a(x+2h) - 2 a(x+h) + a(x)\|_H 
\\
\\
\ds \leq  \int_0^{\|h\|^2_H} e^{-\lambda t} \|\nabla_HT(t)f(x+2h) - 2 \nabla_HT(t)f(x+h) + \nabla_HT(t)f(x )\|_H dt
\\
\\
\ds \leq 4 \int_0^{\|h\|^2_H}  e^{-\lambda t}  c(t) \|f\|_{\infty} dt, 
\end{array}$$
 and setting $c_0:= \sup_{t>0}  t^{1/2} c(t)$ we obtain 
$$\| a(x+2h) - 2 a(x+h) + a(x)\|_H \leq 2c_0 \|f\|_{\infty} \|h\|_H, \quad x\in X. $$
From the obvious equalities 
$$\langle \nabla_HT(t)f(x+2h) - \nabla_HT(t)f(x+h), k\rangle_H = \int_0^1 D^2_HT(t)f(x+ (1+\sigma)h)(h, k)\,d\sigma,$$
$$\langle \nabla_HT(t)f(x+h) - \nabla_HT(t)f(x), k\rangle_H = \int_0^1 D^2_HT(t)f(x+  \sigma h)(h, k)\,d\sigma, \quad k\in H, \; x\in X, $$
we obtain, using   \eqref{stimagradienteH}(iii)
$$\begin{array}{l}
\ds   |\langle \nabla_HT(t)f(x+2h) - 2  \nabla_HT(t)f(x+h)  +  \nabla_HT(t)f(x), k\rangle _H| 
\\
\\
\ds =  \left| \int_0^1 (D^2_HT(t)f(x+ (1+\sigma)h)- D^2_HT(t)f(x+  \sigma h))(h,k) d\sigma\right| 
 \\
 \\
\ds    \leq \sup_{y\in X} \|D^3_HT(t)f(y)\|_{{\mathcal L}^{(3)}(H)} \|h\|^2_H\|k\|_H \leq (3 +k_3^3) c(t)^3  \|f\|_{\infty}  \|h\|^2_H \|k\|_H,  
\end{array}$$
so that, for every $x\in X$, 
$$\|\nabla_HT(t)f(x+2h) - 2  \nabla_HT(t)f(x+h)  +  \nabla_HT(t)f(x)\|_H \leq c_1^3 t^{-3/2} \|f\|_{\infty}  \|h\|^2_H, \quad t>0, $$
with $c_1= (3 +k_3^3) c_0^3$, 
and therefore
$$\begin{array}{l}
\| b(x+2h) - 2 b(x+h) + b(x)\|_H 
\\
\\
\ds  \leq \int_{\|h\|^2}^{\infty}  e^{-\lambda t} \|\nabla_HT(t)f(x+2h) - 2 \nabla_HT(t)f(x+h) + \nabla_HT(t)f(x )\|_H dt
\\
\\
\leq \ds  c_1 \int_{\|h\|^2}^{\infty}  e^{-\lambda t} t^{-3/2} dt \, \|f\|_{\infty}  \|h\|^2_H 
\\
\\
\leq 2c_1 \|f\|_{\infty}  \|h\|_H . 
\end{array}$$
Summing up, we obtain 
$$\| \nabla_H u(x+2h) - 2\nabla_Hu(x+h) + \nabla_H u(x)\|_H \leq (2c_0 + 2c_1)\|f\|_{\infty} \|h\|_H, $$
and the statement follows. 
 \end{proof}

A similar procedure yields maximal H\"older regularity results for the mild solutions to  evolution problems such as \eqref{Cauchy}, 
namely for  the functions given by \eqref{v}, 
for suitable $f$ and $g$. Precisely, we consider the function spaces defined as follows. 

\begin{Definition}
Let $Y$ be any Banach space. For $\alpha\in (0,1)$ we denote by $C^{0,\alpha}_H([0,T]\times X;Y)$  the space of the functions $g\in C_b([0,T]\times X; Y)$ such that $g(t, \cdot)\in C^{\alpha}_H(X; Y)$ for every $t\in [0, T]$, and 
$$\|g\|_{C^{0,\alpha}_H([0,T]\times X; Y)} := \sup_{t\in [0,T]}\|g(t, \cdot)\|_{C^{\alpha}_H(X; Y)} <+\infty. $$
If $Y=\R$ we set $C^{0,\alpha}_H([0,T]\times X;\R) =  C^{0,\alpha}_H([0,T]\times X )$. Moreover, we denote by $C^{0,2+\alpha}_H([0,T]\times X )$ the subspace of $C_b([0,T]\times X)$ consisting of the functions $g$ such that $g(t, \cdot )\in C^{2+\alpha}_H(X)$ for every $t\in [0, T]$, and 
$$\|g\|_{C^{0,2+\alpha}_H([0,T]\times X )} := \sup_{t\in [0,T]}\|g(t, \cdot)\|_{C^{2+\alpha}_H(X )} <+\infty. $$
\end{Definition}

\begin{Theorem}
\label{Schauderpara}
Let $f\in C^{2+\alpha}_H(X)$, $g\in C^{0,\alpha}_H([0,T]\times X)$ with $\alpha \in (0, 1)$, and let $v$ be defined  by \eqref{v}.  Then  $v\in C^{0,2+\alpha}_H([0,T]\times X )$, and there is $C= C(T)>0$, independent of $f$ and $g$, such that 
\begin{equation}
\label{Schauderevol}
\|v\|_{C^{0,2+\alpha}_H([0,T]\times X )} \leq C( \|f\|_{C^{2+\alpha}_H(X)} + \|g\|_{C^{0,\alpha}_H([0,T]\times X )} ). 
\end{equation}
\end{Theorem}
\begin{proof}
We already know that $(t,x)\mapsto T(t)f(x)$ is in $C^{0,2+\alpha}_H([0,T]\times X )$, by Lemma \ref{Le:T(t)holder}. So, we consider the function
$$v_0(t,x):= \int_0^t T(s)g(t-s, \cdot)(x)ds, \quad t\in [0,T], \; x\in X. $$
The same arguments used in the proof of Theorem  \ref{Schauderell} show that 
 $v(t, \cdot)\in C^2_H(X)$ for every $t\in [0,T]$,   that 
$$D^2_Hv_0(t,x) = \int_0^t D^2_HT(s)g(t-s, \cdot)(x)ds, \quad t\in [0,T], \; x\in X, $$
and that there is $C= C(T)>0$, independent of $g$, such that 
$$\|v_0(t, \cdot)\|_{ C^2_H(X)} \leq C \|g\|_{C^{0,\alpha}_H([0,T]\times X )} . $$
Let us prove that $D^2_Hv_0$ is continuous at any $(t_0, x_0)$. 
If $t>t_0$, $x\in X$,  we split
\begin{equation}
\label{splitting}
\begin{array}{l}
\| D^2_Hv_0(t,x)- D^2_Hv_0(t_0,x_0) \|_{{\mathcal L}^{(2)}(H)}    \leq 
\\
\\  \ds \leq  \int_0^{t_0} \|D^2_HT(s)g(t-s, \cdot)(x) - D^2_HT(s)g(t_0-s, \cdot)(x_0))\|_{{\mathcal L}^{(2)}(H)}ds 
\\
\\
 \hspace{4mm} \ds + \int_{t_0}^t \|D^2_HT( s)g(t-s, \cdot)(x)\|_{{\mathcal L}^{(2)}(H)}ds 
\\
\\
 =:    I_1(t,x) + I_2(t,x). 
\end{array}
\end{equation}
Estimate \eqref{stimaderivatealpha}(i) yields
$$I_2(t,x) \leq \int_{t_0}^t \frac{C_{2,\alpha}}{s^{1-\alpha/2}} ds \sup_{0\leq r\leq T}\|g(r, \cdot)\|_{C^{\alpha}_H(X)}, $$
so that $\lim_{t\to t_0^+, x\to x_0} I_2(t,x) =0$. Concerning $I_1(t,x)$, 
for every $s\in [0,t_0]$ and $h$, $k\in H$, formulae \eqref{funzionederseconde} and  \eqref{legge} yield 
$$\begin{array}{l}
|(D^2_HT(s)g(t-s, \cdot)(x) - D^2_HT(s)g(t_0-s, \cdot)(x_0))(h,k)| \leq 
\\
\\
\leq  \ds c(s)^2 \bigg( \int_X |g(t-s, e^{-s}x + \sqrt{1-e^{-2s}}y) - g(t_0-s, e^{-s}x_0 + \sqrt{1-e^{-2s}}y)|^2 \gamma (dy)\bigg)^{1/2} \cdot 
\\
\\
\hspace{4mm} \cdot  \| \hat{h} \hat{k} - \langle h, k\rangle_H\|_{L^2(X, \gamma)}
\\
\\
\leq \ds c(s)^2 \bigg( \int_X |g(t-s, e^{-s}x + \sqrt{1-e^{-2s}}y) - g(t_0-s, e^{-s}x_0 + \sqrt{1-e^{-2s}}y)|^2 \gamma (dy)\bigg)^{1/2} 
\cdot 
\\
\\
\hspace{4mm} \cdot  (k_4^2 +1)\|h\|_H\|k\|_H
\end{array}$$
and since $g$ is continuous and bounded, by the Dominated Convergence Theorem we get 
$$\lim_{t\to t_0, x\to x_0}\|D^2_HT(s)g(t-s, \cdot)(x) - D^2_HT(s)g(t_0-s, \cdot)(x_0)\|_{{\mathcal L}^{(2)}(H)} =0. $$
Moreover, estimate \eqref{stimaderivatealpha}(i) yields
 $$\|D^2_HT(s)g(t-s, \cdot)(x) - D^2_HT(s)g(t_0-s, \cdot)(x_0))\|_{{\mathcal L}^{(2)}(H)} \leq \frac{2C_{2, \alpha}}{s^{1-\alpha/2}}\sup_{0\leq r\leq T}
 \|g(r, \cdot)\|_{C^{\alpha}_H(X)}, \quad 0<s<t. $$
 Therefore, still by the Dominated Convergence Theorem, $\lim_{t\to t_0^+, x\to x_0} I_1(t,x) =0$.  Summing up, we get $\lim_{t\to t_0^+, x\to x_0}
 D^2v_0(t,x) = D^2v_0(t_0, x_0)$. 
 If $t<t_0$, changing the roles of $t$ and $t_0$ in the splitting \eqref{splitting}, we obtain $\lim_{t\to t_0^-, x\to x_0}
 D^2v_0(t,x) = D^2v_0(t_0, x_0)$, and continuity of $D^2v_0$ is proved. 

To prove that $ D^2v_0(t, \cdot) \in C^{\alpha}_H(X, {\mathcal L}^{(2)}(H))$ for every $t\in [0,T]$ we argue as in  Theorem  \ref{Schauderell}, 
namely we split $ D^2v_0(t, \cdot)(x+h) -  D^2v_0(t, \cdot) = a(x+h) -a(x) + b(x+h)-b(x)$, where now
$$a(y) = \int_0^{\min \{t, \|h\|^2\}} D^2_H T(s)g(t-s, \cdot)(y)\,ds, \quad b(y) = \int_{\min \{t, \|h\|^2\}}^t  D^2_H T(s)g(t-s, \cdot)(y)\,ds, \quad y\in X, $$
and we proceed as in the proof  of Theorem  \ref{Schauderell}, to get 
$$[D^2v_0(t, \cdot) ]_{C^{ \alpha}_H(X, {\mathcal L}^{(2)}(H))} \leq   \left( \frac{4C_{2,\alpha}}{\alpha} +  \frac{2C_{3,\alpha}}{1-\alpha} \right) \sup_{r\in [0,T] }\|g(r, \cdot)\|_{C^{\alpha}_H(X)}. $$
\end{proof}          

\section{Open problems and bibliographical remarks}
\label{section_biblio}

Although many of our proofs rely on typical arguments from interpolation theory, interpolation spaces are not explicitly mentioned.  
If $X=\R^d$, Schauder theorems for non-degenerate Ornstein-Uhlenbeck operators were first proved in \cite{DPL}, relying on other interpolation techniques. 
It was shown that for every $f\in C^{\alpha}_b(\R^d)$ the function $R(\lambda, L)f$ defined in  \eqref{risolvente}  is the unique bounded classical solution to \eqref{eq_risolvente}, that  its second order derivatives  belong to the interpolation space 
$$(C_b(\R^d), D(L))_{\alpha/2, \infty} = \{ f\in C_b(\R^d):\; \sup_{t>0} t^{-\alpha/2}\|T(t)f-f\|_{\infty} <+\infty\} ,  $$
where $T(t)$ is the corresponding Ornstein-Uhlenbeck semigroup, 
and the latter space was characterized as 
$$ \{f\in C^{\alpha}_b(\R^d): \; \text{sup}_{t>0} t^{-\alpha/2}\|f(e^{-t}\cdot ) - f\|_{\infty} <+\infty\}. $$
A similar characterization is open in infinite dimension. 
Even the simpler  characterization
\begin{equation}
\label{interp}
(C_b(X), C^1_H(X))_{\alpha, \infty} = C^{\alpha}_H(X), \quad 0<\alpha <1, 
\end{equation}
is not clear in general Banach spaces. In the next lemma we only prove  embeddings, through (by now) standard methods. 

\begin{Lemma}
For every $\alpha\in (0,1)$ we have 
$$\begin{array}{ll}
(i) & (C_b(X), C^1_H(X))_{\alpha, \infty} \subset  C^{\alpha}_H(X),  
\\
\\
(ii) & (C_b(X), D(L))_{\alpha/2, \infty} \subset C^{\alpha}_H(X) . 
\end{array}$$
\end{Lemma}
\begin{proof}
We recall that, given two Banach spaces $\mathcal Y\subset \mathcal X$ and $\alpha\in (0, 1)$, the interpolation space $(\mathcal X, \mathcal Y)_{\alpha, \infty}$ consists of all $u\in \mathcal X$ such that $ \|u\|_{ (\mathcal X, \mathcal Y)_{\alpha, \infty}} := \sup_{t>0} t^{-\theta} K(t, u)  <+\infty$, where 
$K(t,u):= \inf\{ \|a\|_{\mathcal X} + t\|b\|_{\mathcal Y}: \; u=a+b, \; a\in \mathcal X, \; b\in \mathcal Y\}$. We also recall that $(\mathcal X, \mathcal Y)_{\alpha, \infty}\subset \mathcal X $, with continuous embedding. 

Let $u\in (C_b(X), C^1_H(X))_{\alpha, \infty}$. For every decomposition $u= a+b$, with $a\in C_b(X)$, $b\in C^1_H(X)$, we have 
$$|u(x+h) - u(x)| \leq |a(x+h)-a(x)| + |b(x+h) - b(x)| \leq 2\|a\|_{\infty} + \|\nabla_Hb\|_{\infty} \|h\|_H, \quad x\in X, \; h\in H, $$
so that, taking the infimum over all such  decompositions, 
$$|u(x+h) - u(x)| \leq 2K(\|h\|_H, u ) \leq 2 \|h\|_H^{\alpha} \|u\|_{(C_b(X), C^1_H(X))_{\alpha, \infty}}, \quad x\in X, \; h\in H, $$
and (i) follows. 

To prove  statement (ii) we use \eqref{stimagradienteH}(i), that yields, for every $u\in D(L)$ and $\lambda >0$,  $x\in X$, 
$$\begin{array}{lll}
\| \nabla_H u(x) \|_H & \leq & \ds   \int_0^{\infty} e^{-\lambda t} \|\nabla_H T(t)(\lambda u - Lu)(x) \|_H dt 
\\
\\
& \leq  & \ds \int_0^{\infty} e^{-\lambda t} c(t)dt \, (\lambda \|u\|_{\infty } + \|Lu\|_{\infty}) 
\\
\\
&  \leq   &  c_0 \Gamma(1/2) (\lambda ^{1/2}\|u\|_{\infty } + \lambda ^{-1/2}\|Lu\|_{\infty}), 
\end{array}$$
where $c_0= \sup_{t>0} t^{1/2}c(t)$.  Taking the minimum over $\lambda$ we get 
$$ \sup_{x\in X} \| \nabla_H u(x) \|_H \leq C\|u\|_{\infty }^{1/2} \|Lu\|_{\infty}^{1/2}, $$
for some $C>0$, independent of $u$. This implies  that the space $C^1_H(X)$ belongs to the class $J_{1/2}$ between $C_b(X)$ and $D(L)$ (e.g., \cite[Sect. 1.10.1]{T}). The Reiteration Theorem  (\cite[Sect. 1.10.2]{T}) yields 
$$(C_b(X), D(L))_{\alpha/2, \infty} \subset (C_b(X); C^1_H(X))_{\alpha, \infty},  $$
and (ii) follows from (i). 
\end{proof}
 
Going back to \eqref{interp}, in the case where $X$ is a Hilbert space and $\gamma$ is  non-degenerate, the  similar equality
$$(BUC(X), BUC^1_H(X))_{\alpha, \infty} = C^{\alpha}_H(X)\cap BUC(X) $$
was stated  in \cite{CDP1}.

Concerning Schauder estimates in infinite dimension, if $X$ is an infinite dimensional Hilbert space, smoothing Ornstein-Uhlenbeck semigroups such as 
$$T(t)f(x) = \int_X f(e^{tA}x + y){\mathcal N}_{0,Q_t}(dy)$$
were considered in \cite{C,CDP}, under the assumptions that $A$ is the infinitesimal generator of a strongly continuous semigroup $e^{tA}$ in $X$, $Q\in {\mathcal L}(X)$ is a self-adjoint positive operator, the operators $Q_t:= \int_0^t e^{sA}Qe^{sA^*}ds $ have finite trace for every $t>0$, $e^{tA}(X)\subset Q_t^{1/2}(X)$ for every $t$, and moreover $\sup_{t>0}t^{1/2} \|Q_t^{-1/2}e^{tA}\|_{\mathcal L(X)}<+\infty$. The generator of $T(t)$ is a realization of the operator 
$$\mathcal L u(x)= \frac{1}{2}{\text Tr}(QD^2u(x)) + \langle x, A^*\nabla u(x)\rangle$$
and  $T(t)$ is a smoothing operator in all directions, not only along a subspace. In this case,  a Schauder theorem in the usual H\"older spaces holds: namely, if $f$ is any bounded function belonging to $C^{\alpha}(X)$ for some $\alpha\in (0, 1)$, then the function 
\begin{equation}
\label{1}
u(x) = \int_0^{\infty} e^{-\lambda t}T(t)f(x)dt
\end{equation}
 belongs to $C^2(X)$, it has  bounded first and second order derivatives,  $D^2u \in C^{\alpha}(X, {\mathcal L}^{(2)}(X))$.  This was proved in \cite{CDP} in the case $Q=I$ and in \cite[Ch. 5]{C} in the case that $T(t)$ is the transition semigroup of a suitable linear stochastic PDE with  $X= L^2(\Omega)$, $\Omega$ being an open bounded subset of $\R^d$ with smooth boundary.  

 We would like to remind that there are relevant situations in which Schauder estimates cannot be proved for Hilbert spaces, but only for Banach spaces. This is the case considered in \cite{cdp}, where the transition semigroup $T(t)$ associated with a class of stochastic reaction-diffusion equations defined on a bounded interval 
$[0,1]$,  with polynomially growing coefficients, is studied in the space $X=C([0,1])$. Actually, for that class of equations the analysis of $T(t)$  in $X=L^2(0,1)$ is considerably more delicate than in $X=C([0,1])$ and it is not possible to prove that when  $f \in\,C^\alpha(L^2(0,1))$, for some $\alpha \in\,(0,1)$, the function $u$ defined in \eqref{1} belongs to $C^2(L^2(0,1))$. Notice, in particular, that  working in $C([0,1])$ prevents from using the interpolatory identity \eqref{interp}.

 Under assumptions similar to \cite{CDP} a related result is in \cite{ABP}, where the space $L^{\infty}(X, \gamma)$   is considered instead of $C_b(X)$. Regularity results were stated in terms of the  spaces $\{ f\in L^{\infty}(X, \gamma):\; \sup_{t>0} t^{-\alpha/2}\|T(t)f-f\|_{\infty} <+\infty\}$, called  $S^{\alpha}$ and endowed with their natural norm
$$\|f\|_{\infty} +  \sup_{t>0} \frac{\|T(t)f -f\|_{\infty}}{t^{\alpha/2}} . $$
However, since $T(t)$ is strong Feller, we have $S^{\alpha}= (C_b(X), D(L))_{\alpha/2, \infty}$, with equivalence of the respective norms. 
In \cite{ABP} it is proved that if $f\in S^{\alpha}$, then $u$ and  its first and second order derivatives along any direction  belong to $S^{\alpha}$.

Schauder type theorems for the Gross Laplacian and of some of its perturbations were established in \cite{CDP1}. Here, the semigroup is given by 
$$S(t)f(x) = \int_X \varphi(x+\sqrt{t}y)\gamma(dy), $$
where $\gamma$ is again a centered non-degenerate Gaussian measure in a separable Hilbert space $X$. In contrast with Ornstein-Uhlenbeck semigroups, $S(t)$ is strongly continuous in $BUC(X)$; similarly to our Ornstein-Uhlenbeck semigroup $S(t)$ is not strong Feller, and it is smoothing only along the directions of the Cameron-Martin space. The infinitesimal generator of $S(t)$ in $BUC(X)$ is a realization $A$ of the operator 
$${\mathcal A}u(x) = \frac{1}{2}\text{Trace}\,D^2u(x), $$
in the space $BUC(X)$. 
A result similar to Theorem 1 was stated, when $f\in C^{\alpha}_H(X)\cap BUC(X) $ (the latter space is called $C^{\alpha}_Q(X)$ in \cite{CDP1}, $Q$ being the covariance of $\gamma$). Moreover, 
 in \cite{ABP} it was proved that if $f\in S^{\alpha}$, where now
 $$ S^{\alpha}: =  \{ g\in L^{\infty}(X):\; \sup_{t>0}  t^{-\alpha/2}\|S(t)g-g\|_{\infty} <+\infty  \}, $$
 then for every $\lambda >0$ the function $u(x)=\int_X e^{-\lambda t}S(t)f(x) ds$ possesses first and second order derivatives along the elements of any orthonormal basis of $X$ consisting of eigenvalues of $Q$, and they  belong to $ S^{\alpha} $. 
  
\section{Acknowledgements}

The first  author was partially supported the NSF Research Grant DMS-1712934 (2017-2020), ``Analysis of Stochastic Partial Differential Equations with Multiple Scales". The second author is a member of GNAMPA-INDAM, and was partly supported by the PRIN Grant 2015233N54 ``Deterministic and stochastic evolution equations". 



\begin{thebibliography}{99}

 \bibitem{ABP} {\sc S.A. Athreya, R.F. Bass, E.A. Perkins}, {\it H\"older norm estimates for elliptic operators on finite and infinite dimensional spaces}, Trans. Amer. Math. Soc. {\bf 357} (2005), 5001-5029. 

 \bibitem{Boga} {\sc V.I. Bogachev}, {\rm Gaussian Measures}, American Mathematical Society, 1998.

\bibitem{CDP}  {\sc P. Cannarsa, G. Da Prato}, {\it Schauder estimates for Kolmogorov equations in Hilbert spaces}, Progress in elliptic and parabolic partial differential equations (Capri, 1994), 100-111, Pitman Res. Notes Math. Ser., 350, Longman, Harlow, 1996. 

\bibitem{CDP1}  {\sc P. Cannarsa, G. Da Prato}, {\it Infinite-dimensional elliptic equations with H\"older-continuous coefficients}, Adv. Differential Equations {\bf 1} (1996),  425-452. 

\bibitem{C} {\sc S. Cerrai},  {\it Second Order PDEs in Finite and Infinite Dimension. A probabilistic approach}, Lecture Notes in Mathematics, 1762. Springer-Verlag, Berlin, 2001.

\bibitem{cdp} {\sc S. Cerrai, G. Da Prato},  {\it Schauder estimates for elliptic equations in Banach spaces associated with stochastic reaction-diffusion equations}, Journal of Evolution Equations {\bf 12} (2012), 83-98.

 \bibitem{DPL} {\sc G. Da Prato, A. Lunardi}, {\it On the Ornstein-Uhlenbeck operator in spaces of continuous functions}, J. Funct. Anal. {\bf 131} (1995), 94-114. 
  
 \bibitem{L} {\sc  A. Lunardi}, {\it  Schauder estimates for a class of parabolic operators with unbounded coefficients in
$\R^n$}, Ann. Scuola Norm. Sup. Pisa (4) {\bf 24} (1997), no. 1, 133Ð164.
  
  


 \bibitem{T} {\sc H. Triebel} {\it  Interpolation Theory, Function Spaces, Differential Operators}, North-Holland, Amsterdam, 1978. 

 \end{thebibliography}
\end{document}